\documentclass[12pt]{amsart}
\usepackage{txfonts}
\usepackage{amscd,amssymb,mathabx,subfigure}
\usepackage[arrow,matrix,graph,frame,poly,arc,tips]{xy}
\usepackage{graphicx,calrsfs}
\usepackage[colorlinks,plainpages,backref,urlcolor=blue]{hyperref}
 
\usepackage{mdwlist}
\usepackage{multirow}
\usepackage{enumerate}
\usepackage[shortlabels]{enumitem}
\usepackage{verbatim}
\usepackage[abs]{overpic}

\numberwithin{equation}{section}
\allowdisplaybreaks
\usepackage{amsfonts,amssymb,verbatim,amsmath,amsthm,latexsym,textcomp,amscd}
\usepackage{latexsym,amsfonts,amssymb,epsfig,verbatim}
\usepackage{amsmath,amsthm,amssymb,latexsym,graphics,textcomp, hyperref}
\usepackage{color}
\usepackage{url}
\usepackage{psfrag}

\begin{document}

\newtheorem{theorem}{Theorem}[section]
\newtheorem{prop}[theorem]{Proposition}
\newtheorem{lemma}[theorem]{Lemma}
\newtheorem{lem}[theorem]{Lemma}
\newtheorem{cor}[theorem]{Corollary}
\newtheorem{definition}[theorem]{Definition}
\newtheorem{conj}[theorem]{Conjecture}
\newtheorem{claim}[theorem]{Claim}
\newtheorem{qn}[theorem]{Question}
\newtheorem{defn}[theorem]{Definition}
\newtheorem{defth}[theorem]{Definition-Theorem}
\newtheorem{obs}[theorem]{Observation}
\newtheorem{rmk}[theorem]{Remark}
\newtheorem{ans}[theorem]{Answers}
\newtheorem{slogan}[theorem]{Slogan}
\newtheorem{eg}[theorem]{Example}
\newtheorem{assume}[theorem]{Assumption}
\newcommand{\bluecomment}[1]{\textcolor{blue}{#1}}

\newcommand{\thmref}[1]{Theorem~\ref{#1}}
\newcommand{\corref}[1]{Corollary~\ref{#1}}
\newcommand{\lemref}[1]{Lemma~\ref{#1}}
\newcommand{\propref}[1]{Proposition~\ref{#1}}
\newcommand{\defref}[1]{Definition~\ref{#1}}
\newcommand{\egref}[1]{Example~\ref{#1}}
\newcommand{\rmkref}[1]{Remark~\ref{#1}}
\newcommand{\exerref}[1]{Exercise~\ref{#1}}
\newcommand{\itemref}[1]{\textnormal{(\ref{#1})}}
\newcommand{\itemrefstar}[1]{(\ref*{#1})}
\newcommand{\figref}[1]{Figure~\ref{#1}}
\newcommand{\secref}[1]{Section~\ref{#1}}

\newcommand{\boundary}{\partial}
\newcommand{\dirac}{\textrm{Dirac}}
\newcommand{\hhat}{\widehat}
\newcommand{\C}{{\mathbb C}}
\newcommand{\bH}{{\mathbb H}}
\newcommand{\sqQ}{{\sqrt{\mathbb Q}}}
\newcommand{\Ga}{{\Gamma}}
\newcommand{\gam}{{\Gamma}}
\newcommand{\G}{{\Gamma}}
\newcommand{\s}{{\Sigma}}
\newcommand{\PSL}{{\mathrm{PSL}_2 (\mathbb{C})}}
\newcommand{\SL}{{\mathrm{SL}_2 (\mathbb{C})}}
\newcommand{\pslc}{{\mathrm{PSL}_2 (\mathbb{C})}}
\newcommand{\pslr}{{\mathrm{PSL}_2 (\mathbb{R})}}
\newcommand{\pslz}{{\mathrm{PSL}_2 (\mathbb{Z})}}
\newcommand{\pslq}{{\mathrm{PSL}_2 (\mathbb{Q})}}
\newcommand{\pslo}{{\mathrm{PSL}_2 (\OO)}}
\newcommand{\slq}{{\mathrm{SL}_2 (\mathbb{Q})}}
\newcommand{\slr}{{\mathrm{SL}_2 (\mathbb{R})}}
\newcommand{\slz}{{\mathrm{SL}_2 (\mathbb{Z})}}
\newcommand{\slf}{{\mathrm{SL}_2 (F)}}
\newcommand{\slo}{{\mathrm{SL}_2 (\mathbb{\OO})}}
\newcommand{\so}{{\mathrm{SO}}}
\newcommand{\Gr}{{\mathcal G}}
\newcommand{\integers}{{\mathbb Z}}
\newcommand{\natls}{{\mathbb N}}
\newcommand{\ratls}{{\mathbb Q}}
\newcommand{\reals}{{\mathbb R}}
\newcommand{\proj}{{\mathbb P}}
\newcommand{\lhp}{{\mathbb L}}
\newcommand{\tube}{{\mathbb T}}
\newcommand{\cusp}{{\mathbb P}}
\newcommand{\AAA}{{\mathcal A}}
\newcommand\HHH{{\mathbb H}}
\newcommand{\BB}{{\mathcal B}}
\newcommand\CC{{\mathcal C}}
\newcommand\DD{{\mathcal D}}
\newcommand\EE{{\mathcal E}}
\newcommand\FF{{\mathcal F}}
\newcommand\GG{{\mathcal G}}
\newcommand\HH{{\mathcal H}}
\newcommand\II{{\mathcal I}}
\newcommand\JJ{{\mathcal J}}
\newcommand\KK{{\mathcal K}}
\newcommand\LL{{L}}
\newcommand\MM{{\mathcal M}}
\newcommand\NN{{\mathcal N}}
\newcommand\OO{{\mathcal O}}
\newcommand\PP{{\mathcal P}}
\newcommand\QQ{{\mathcal Q}}
\newcommand\RR{{\mathcal R}}
\newcommand\SSS{{\mathcal S}}
\newcommand\TT{{\mathcal T}}
\newcommand\UU{{\mathcal U}}
\newcommand\VV{{\mathcal V}}
\newcommand\WW{{\mathcal W}}
\newcommand\XX{{\mathcal X}}
\newcommand\YY{{\mathcal Y}}
\newcommand\ZZ{{\mathcal Z}}
\newcommand\CH{{\CC\Hyp}}
\newcommand{\Chat}{{\hat {\mathbb C}}}
\newcommand\MF{{\MM\FF}}
\newcommand\PMF{{\PP\kern-2pt\MM\FF}}
\newcommand\ML{{\MM\LL}}
\newcommand\PML{{\PP\kern-2pt\MM\LL}}
\newcommand\GL{{\GG\LL}}
\newcommand\Pol{{\mathcal P}}
\newcommand\half{{\textstyle{\frac12}}}
\newcommand\Half{{\frac12}}
\newcommand\Mod{\operatorname{Mod}}
\newcommand\Area{\operatorname{Area}}
\newcommand\Isom{\operatorname{Isom}}
\newcommand\ep{\epsilon}
\newcommand\Hypat{\widehat}
\newcommand\Proj{{\mathbf P}}
\newcommand\U{{\mathbf U}}
 \newcommand\Hyp{{\mathbf H}}
\newcommand\D{{\mathbf D}}
\newcommand\Z{{\mathbb Z}}
\newcommand\bZ{{\mathbb Z}}
\newcommand\R{{\mathbb R}}
\newcommand\Q{{\mathbb Q}}
\newcommand\E{{\mathbb E}}
\newcommand\Emu{{\mathbb E_\mu}}
\newcommand\EXH{{ \EE (X, \HH_X )}}
\newcommand\EYH{{ \EE (Y, \HH_Y )}}
\newcommand\GXH{{ \GG (X, \HH_X )}}
\newcommand\GYH{{ \GG (Y, \HH_Y )}}
\newcommand\ATF{{ \AAA \TT \FF }}
\newcommand\PEX{{\PP\EE  (X, \HH , \GG , \LL )}}
\newcommand{\lct}{\Lambda_{CT}}
\newcommand{\lel}{\Lambda_{EL}}
\newcommand{\lgel}{\Lambda_{GEL}}
\newcommand{\lre}{\Lambda_{\mathbb{R}}}
\newcommand{\sas}{{S\alpha S}}
\newcommand{\lag}{{\Lambda_\Ga}}
\newcommand{\lagg}{{\Lambda_G^{(2)}}}
\newcommand{\laggr}{{\Lambda_G^{(2)}\times \R}}
\newcommand{\stwo}{{S^{(2)}}}
\newcommand{\fex}{{f_{ex}}}
\newcommand{\fexr}{{f_{exr}}}

\newcommand\til{\widetilde}
\newcommand\length{\operatorname{length}}
\newcommand\tr{\operatorname{tr}}
\newcommand\gesim{\succ}
\newcommand\lesim{\prec}
\newcommand\simle{\lesim}
\newcommand\simge{\gesim}
\newcommand{\simmult}{\asymp}
\newcommand{\simadd}{\mathrel{\overset{\text{\tiny $+$}}{\sim}}}
\newcommand{\ssm}{\setminus}
\newcommand{\diam}{\operatorname{diam}}
\newcommand{\pair}[1]{\langle #1\rangle}
\newcommand{\T}{{\mathbf T}}
\newcommand{\I}{{\mathbf I}}
\newcommand{\mups}{{\mu^{PS}}}
\newcommand{\mubm}{{\mu^{BM}}}
\newcommand{\mubms}{{\mu^{BMS}}}
\newcommand{\mubmr}{{\mu^{BMR}}}
\newcommand{\mubmrm}{{\mu^{BMRm}}}
\newcommand{\musk}{{\mu^{S}}}
\newcommand{\sro}{{\s_r(o)}}

\newcommand{\tw}{\operatorname{tw}}
\newcommand{\comm}{\operatorname{Comm_\pslr}}
\newcommand{\comme}{\operatorname{Comm}}
\newcommand{\base}{\operatorname{base}}
\newcommand{\trans}{\operatorname{trans}}
\newcommand{\rest}{|_}
\newcommand{\bbar}{\overline}
\newcommand{\UML}{\operatorname{\UU\MM\LL}}
\newcommand{\EL}{\mathcal{EL}}
\newcommand{\qle}{\lesssim}
\newcommand{\erfex}{{\E_r(\fex)}}

\newcommand\Gomega{\Omega_\Gamma}
\newcommand\nomega{\omega_\nu}
\newcommand\qcg{{QC(\Lambda_G)}}

\newcommand\convd{\stackrel{d}{\rightarrow}}
\newcommand\convp{\stackrel{p}{\rightarrow}}
\newcommand \convas{\stackrel{a.s.}{\rightarrow}}
\newcommand\eqd{\stackrel{d}{=}}

\title[Commutators and commensurators]{Commutators, commensurators, and $\mathrm{PSL}_2(\integers)$}

\author[T. Koberda]{Thomas Koberda}
\address{Department of Mathematics, University of Virginia, Charlottesville, VA 22904-4137, USA}
\email{thomas.koberda@gmail.com}
\urladdr{http://faculty.virginia.edu/Koberda}

\author{Mahan Mj}
\address{School of Mathematics, Tata Institute of Fundamental Research, 1 Homi Bhabha Road, Mumbai 400005, India}

\email{mahan@math.tifr.res.in}
\email{mahan.mj@gmail.com}
\urladdr{http://www.math.tifr.res.in/~mahan}

\thanks{The first author is partially supported by an Alfred P. Sloan Foundation Research Fellowship and by NSF Grants DMS-1711488
and DMS-2002596.
The second author is  supported in part by  the Department of Atomic Energy, Government of India, under project no.12-R\&D-TFR-5.01-0500, as
also by an endowment of the Infosys Foundation,
a DST JC Bose Fellowship, Matrics research project grant  MTR/2017/000005,
and CEFIPRA  project No. 5801-1.}

\date{\today}

\begin{abstract}
Let $H<\pslz$ be a finite index normal subgroup which is contained in a principal
congruence subgroup, and let $\Phi(H)\neq H$ denote a term of the lower central series or the derived series of $H$.
In this paper, we prove that the commensurator
of $\Phi(H)$ in $\pslr$ is discrete. We thus obtain a natural family of thin subgroups of $\pslr$ whose
commensurators are discrete, establishing some cases of a conjecture of Shalom.
\end{abstract}

\maketitle


\section{Introduction} The Commensurability Criterion for Arithmeticity due to Margulis \cite{margulis} \cite[16.3.3]{morris-book}
says that among irreducible lattices in  semi-simple Lie groups, arithmetic lattices are characterized  as those that have dense 
commensurators.
During the past decade, Zariski dense discrete subgroups of infinite covolume in semi-simple Lie groups,
also known as {\it thin subgroups} \cite{sarnak-thin}, have gained a lot of attention.
Heuristically, thin subgroups should be regarded as non-arithmetic, even though the essential difference between a thin group and a lattice
is that the former has infinite covolume in the ambient Lie group.
A question attributed to Shalom  makes this heuristic precise in the following way:

\begin{qn} \cite{llr}
	\label{mainq}
Suppose that
$\Ga$
is a thin subgroup of a semisimple Lie group
$G$.   Is
the commensurator of
$\Ga$ discrete?
\end{qn}
Let $X$ be the symmetric space of non-compact type associated to $G$ and 
 $\partial X$ denote its Furstenberg boundary. Question \ref{mainq} has been answered affirmatively in the following cases:
\begin{enumerate}
	\item Let $\lag$ denote the limit of $\Ga$ on $\partial X$. If $\lag \subsetneq \partial X$, then the answer to Question \ref{mainq} is affirmative (\cite{greenberg} for $G=\pslc$ and  \cite{mahan-commens} in the general case).
	\item If $G=\pslc$ and $\Ga$ is finitely generated, then the answer to Question \ref{mainq} is affirmative \cite{llr,mahan-commens}.
\end{enumerate}

Question \ref{mainq} thus  remains unaddressed when 
\begin{enumerate}
\item $\lag = \partial X$, and further,
\item $\Ga$ is not a finitely generated subgroup of $\pslc$.
\end{enumerate}

Curiously, Question \ref{mainq} remains open even for thin subgroups $\Ga$ of the simplest non-compact simple Lie group $G=\pslr$, with
$\Ga$ satisfying the condition $\lag = \partial \Hyp^2 = S^1$. It is easy to see that such a $\Ga$ cannot be finitely generated.
To see this last claim, if
$\Ga$ is thin then discreteness implies that $\gam$ must be free or the fundamental group of a closed surface. Since thin groups have
infinite covolume, $\gam$ cannot be a closed surface group, and hence must be free. It follows that $\Sigma=\Hyp^2/\gam$ must be a surface
with boundary or punctures. If $\Sigma$ has infinitely many punctures or boundary components, or if $\Sigma$ has infinite genus, then
clearly $\gam$ is not finitely generated. If $\Sigma$ has finite genus and
only finitely many punctures and no boundary components, then $\Sigma$
has finite volume and so $\gam$ is not thin. If $\Sigma$ has a boundary component
(i.e.~a flaring end) then $\Sigma$ admits a proper geodesically convex
core,
in which case the limit set $\lag$ will be properly contained in $\partial \Hyp^2$.

In this paper, we shall study commensurators of certain natural infinite index subgroups of $\pslz$.
If $\G<G$ is an arbitrary subgroup of a group $G$,
we define \[\mathrm{Comm}_G(\G)=\{g\in G\mid [\G:\G^g\cap \G]\textrm{ and }[\G^g:\G^g\cap \G]<\infty\}.\]
Here, we use exponentiation notation to denote conjugation, so that
$\G^g=g^{-1}\G g$.

\subsection{The main result}
In this paper, we will concentrate on thin subgroups of $\pslr$, which are of particular interest because of their intimate
connections to hyperbolic geometry via the identification \[\pslr\cong\Isom^+(\Hyp^2).\] For an integer $k\geq 2$, we will write
$\G(k)<\pslz$ for the \emph{level $k$ principal congruence subgroup}, which is to say the kernel of the map
\[\pslz\to\mathrm{PSL}_2(\Z/k\Z)\]
given by reducing the entries modulo $k$. 

As a matter of notation, we will write $\pslq\sqQ$ for the projectivization of the set of matrices in
$\mathrm{SL}_2(\R)$ which differ from a matrix in $\mathrm{GL}_2(\Q)$ by
a scalar matrix which is a square root of a rational number. That is, we have $A\in\pslq\sqQ$ if there
is a representative of $A$ in $\mathrm{SL}_2(\R)$, a rational number $q\in\Q$,
and a matrix $B\in\mathrm{GL}_2(\Q)$ such that \[A=\begin{pmatrix}\sqrt{q}&0\\0&\sqrt{q}\end{pmatrix}\cdot B.\]

Since we may also write \[A=\begin{pmatrix}\sqrt{q}&0\\0&\sqrt{q^{-1}}\end{pmatrix}\cdot B'\] for some $B'\in\mathrm{SL}_2(\Q)$ and since
the matrix \[\begin{pmatrix}\sqrt{q}&0\\0&\sqrt{q^{-1}}\end{pmatrix}\] normalizes $\mathrm{SL}_2(\Q)$, it is easy to see that
$\pslq\sqQ$ forms a group which can be viewed as the join
\[\pslq\cdot\left\langle \begin{pmatrix}\sqrt{q}&0\\0&\sqrt{q^{-1}}\end{pmatrix}\mid q\in \Q\right\rangle<\pslr.\]

We need to consider $\pslq\sqQ$ because of  the ambient group where the commensurator lives. We recall (see \cite[p. 92]{morris-book}
or \cite[Ex. 6d]{loeffler}) that  \[\comm(\pslz)=\pslq\sqQ,\] which is dense in $\pslr$.
As a consequence, it is easy to see that if $\G<\pslz$ has finite index then
\[\comm(\G)=\pslq\sqQ.\]

For an arbitrary group $G$, we recall the definition of the \emph{lower central series} and the \emph{derived series} of $G$. For the lower
central series, we define $\gamma_1(G)=G$ and $\gamma_{i+1}(G)=[G,\gamma_i(G)]$. The derived series is defined by $D_1(G)=G$ and
$D_{i+1}(G)=[D_i(G),D_i(G)]$. We will often use the notation $G'=[G,G]$ for the derived subgroup of $G$. Observe that
\[G'=D_2(G)=\gamma_2(G),\] and if $G$ is a free group then for $i\geq 2$ we have that $\gamma_i(G)$ and $D_i(G)$ are both properly
contained in $G$ as infinitely generated characteristic subgroups of infinite index.

The purpose of this article is to establish following result, which answers Question~\ref{mainq}
in the affirmative for, perhaps, the most `arithmetically' defined examples:

\begin{theorem}\label{thm:main}
Let $H\lhd \pslz$ be a finite index normal
subgroup with $H<\G(k)$ for some $k\geq 2$, and let $\Phi(H)$ denote a term in the lower central series or derived
series of $H$. Suppose furthermore
that $\Phi(H)\neq H$, so that $[H:\Phi(H)]=\infty$.
If \[g\in\comm(\Phi(H)),\] then we have that
$g^2\in\pslz$. In particular, $\comm(\Phi(H))$ is discrete.
\end{theorem}

In Theorem~\ref{thm:main}, it is unclear to the authors how to weaken the hypothesis that $H$ be normal in
$\pslz$, as well as the assumption that $H$ be contained in a principal congruence subgroup. We will outline a conjectural picture shortly.

We remark that in order to establish a perfect analogy with Margulis' Arithmeticity Theorem,
one would like to show that $\Phi(H)$ has finite index
in $\comm(\Phi(H))$. This, however, is simply not feasible.
On the one hand, if $H=\G(k)$ for some $k\geq 2$ then $\Phi(H)$ is normal in $\pslz$, whence
\[\pslz<\comm(\Phi(H)).\] On the other hand, the index of $\Phi(H)$ in $\pslz$ is always infinite under
the hypotheses of Theorem~\ref{thm:main}.
Nevertheless, we obtain the following corollary to Theorem~\ref{thm:main}, which is the correct way to mend the analogy with the Arithmeticity
Theorem:
\begin{cor}\label{cor:shalom-cor}
Let $\Phi(H)$ be as in Theorem~\ref{thm:main}. Then $\comm(\Phi(H))$ is commensurable with $\pslz$. In particular, $\comm(\Phi(H))$
contains the normalizer of $\Phi(H)$ with finite index.
\end{cor}

Corollary~\ref{cor:shalom-cor} follows from Theorem~\ref{thm:main} by standard methods from the theory
	of lattices in Lie groups. Indeed, $\Phi(H)$ is normal
in $H$, and
the normalizer of $\Phi(H)$ is obviously contained in $\comm(\Phi(H))$.
Note that $\Sigma=\Hyp^2/H$ is a finite volume hyperbolic orbifold. Theorem~\ref{thm:main} implies that
\[\Hyp^2/\comm(\Phi(H))\] is also a finite volume hyperbolic orbifold, and so is a quotient of $\Sigma$ by a finite group of isometries.
It follows that $\comm(\Phi(H))$ is commensurable with $\pslz$ and contains the normalizer of $\Phi(H)$ with finite index. The reader
may compare these last remarks with Lemma~\ref{normalizer} below, for instance.

In our notation, we will suppress which series for $H$ as well as which term we are considering, since this will not cause particular confusion.
Observe that since $H$ has finite index in $\pslz$ and since $\Phi(H)$ is a non-elementary normal subgroup of $H$, we have that on the
level of limit sets, \[\Lambda_{\pslz}=\Lambda_H=\Lambda_{\Phi(H)}=S^1.\] The reader may consult~\cite{benoistgafa}
for more details on Zariski density
and limit sets, for instance.
We have in particular that the group $\Phi(H)$ falls outside of the purview of extant
results concerning the commensurators of thin subgroups.

The techniques used in the proof of Theorem \ref{thm:main} differ widely from those used in \cite{llr,mahan-commens}. These papers used an action on a topological space and discreteness of the commensurator stemmed from the fact that the commensurator preserved a `discrete geometric pattern' (in the sense of Schwartz, cf.\ \cite{schwartz-pihes,schwartz-inv,biswasmj,mahan-pattern}). In this paper we use an {\it algebraic} action from Chevalley-Weil theory and homological algebra as a replacement in
order to conclude discreteness of the commensurator.
Since the proof is somewhat tricky and consists of a number of
modular pieces, we outline the steps involved (using the notation of Theorem \ref{thm:main} above):

\begin{enumerate}
\item We discuss some basic facts about principal congruence groups in Section \ref{sec:background}, and observe in Section \ref{sec:invariant} that the normalizer of  $\Phi(H)$ lies in $\pslz$.
\item Section \ref{sec:gaschutz} introduces the first technical tool in the paper based on Chevalley-Weil theory. We prove (Lemma \ref{lem:almost-inv}) that if $g\in\pslr$ conjugates  $H<\pslz$ to $H^g\neq H$ and both $H,H^g$ are contained in a common free subgroup $F$,
 then $g$ does not lie in the
commensurator of $\Phi(H)$.
\item Section \ref{sec:intcomm} is  devoted to proving that if $g \in \pslr$ commensurates $\Phi(H)$, then in fact 
$g^2 \in \pslq\sqQ$ (Lemma \ref{lem:integcomm}).
The proof, given in Section \ref{sec:ant}, uses ideas from  invariant trace fields and quaternion division algebras and is number-theoretic
in flavor.
\item In Section \ref{sec:homology} we introduce the last tool, which to the knowledge of the authors is completely novel:
a partial action on homology. We call this partial action a pseudo-action and study it on homology classes carried by cusps to complete the proof of Theorem \ref{thm:main}.
\end{enumerate}

\subsection*{A conjectural picture}
The proof of Theorem~\ref{thm:main} draws on several different areas of mathematics, including hyperbolic geometry, homological algebra, 
noncommutative algebra, and Galois
theory. In order to make all the arguments work, we were forced to adopt the hypothesis that $H$ be contained as a normal subgroup of a principal
congruence subgroup. However, we expect the following more general statements to hold:

\begin{conj}\label{conjpslr}
Let $H<\pslz$ be a finite index subgroup with $b_1(H)\geq 1$. Then $\comm(\Phi(H))$ is discrete provided that $H\neq \Phi(H)$.
More generally, let $K$ be an  infinite index normal subgroup of a lattice $\Ga<\pslr$ such that $|K|=\infty$. Then $\comm(K)$ is discrete.
In all cases, we have that $\comm(K)$ contains the normalizer of $K$ with finite index.
\end{conj}

In the first statement of Conjecture \ref{conjpslr}, the Betti number assumption $b_1(H)\geq 1$
is simply to guarantee that the derived subgroup $[H,H]=H'$ has infinite index in $H$. It would also be a reasonable alternative assumption
to require $H$ to be torsion-free. Normality of the infinite index subgroups is always assumed in order to make the limit set coincide
with the full circle.


\subsection{Preliminaries on principal congruence subgroups}\label{sec:background}

In this subsection, we gather some well-known facts about principal congruence subgroups which will be useful in the sequel. We include
proofs for the convenience of the reader and to keep the discussion as self-contained as possible.

\begin{lem}\label{lem:pctf}
Let $k\geq 2$. Then $\G(k)$ is a free group.
\end{lem}
\begin{proof}
The quotient $\Hyp^2/\pslz$ is the $(2,3,\infty)$ hyperbolic orbifold. It follows that if $g\in\pslz$ has finite order then it is conjugate to
(the image of)
one of the two matrices \[A=\begin{pmatrix}0&-1\\1&0\end{pmatrix}\] or \[B=\begin{pmatrix}1&-1\\1&0\end{pmatrix}\] of orders $2$ and $3$
respectively. Let \[q\colon\pslz\to Q\] be a finite quotient. If $A$ and $B$ do not lie in the kernel of $q$ then $\ker q$ is torsion-free. Indeed,
if $\ker q$ contains a torsion element then this element would be conjugate to either $A$ or $B$, so that normality would imply that $A$ or
$B$ lies in $\ker q$, contrary to the assumption.
Since $A$ and $B$ are clearly nontrivial in $\pslz/\G(k)$, we see that $\G(k)$ must be torsion-free. Since
$\G(k)$ is torsion-free and acts freely and properly discontinuously on $\Hyp^2$, and since $\Hyp^2/\pslz$ is noncompact
and orientable, we have
that $\Hyp^2/\G(k)$ is a noncompact $2$--dimensional manifold with fundamental group $\G(k)$. It follows that $\G(k)$ is free.
\end{proof}

Note that the modular surface $\Hyp^2/\pslz$ has exactly one cusp. As an element of the orbifold fundamental group of the modular surface,
the free homotopy class of the cusp is generated by the matrix \[\begin{pmatrix} 1&1\\ 0&1\end{pmatrix}.\]

\begin{lem}\label{lem:cusp}
Let $k\geq 2$. The hyperbolic manifold $\Hyp^2/\G(k)$ has at least three cusps.
\end{lem}
\begin{proof}
We may reduce to the case where $k$ is a prime, since if $p|k$ is a prime divisor of $k$ then $\G(k)<\G(p)$. The map
\[\pslz\to\mathrm{PSL}_2(\Z/p\Z)\] is surjective (as can be deduced from examining generating sets of $\mathrm{PSL}_2(\Z/p\Z)$),
and its image has order \[p(p+1)(p-1)/2\] when $p$ is odd and order $6$ when $p=2$. The order of the matrix
\[\begin{pmatrix} 1&1\\ 0&1\end{pmatrix}\] in $\mathrm{PSL}_2(\Z/p\Z)$ is exactly $p$, so that we see that there are exactly
\[(p+1)(p-1)/2\] distinct cusps in $\Hyp^2/\G(p)$ when $p\geq 3$ and exactly three cusps when $p=2$. The lemma now follows.
\end{proof}

The following consequence is immediate from elementary surface topology:

\begin{cor}\label{cor:homology}
If $k\geq 2$, then any loop about a cusp in $\Hyp^2/\G(k)$ represents a nontrivial integral homology class.
Moreover, two loops about distinct cusps
in $\Hyp^2/\G(k)$ represent distinct integral homology classes.
\end{cor}
\begin{proof}
For all $k\geq 2$, we have that $\Sigma=\Hyp^2/\G(k)$ has at least three cusps, as follows from Lemma~\ref{lem:cusp}.
Choose a bi-infinite, simple,
oriented geodesic $\gamma$ on $\Sigma$ that travels between two of the cusps of $\Sigma$. If $K$ denotes a small
neighborhood of the union of the cusps, we have that \[H^1(\Sigma,K,\bZ)\cong H_1(\Sigma,\bZ),\] by relativized Poincar\'e duality
(i.e.~Poincar\'e--Lefschetz duality),
and algebraic intersection with $\gamma$ represents a nonzero cohomology class in $H^1(\Sigma,K,\bZ)$. If $\delta$ is a small
oriented loop around one of the cusps of $\Sigma$, then the algebraic intersection pairing of $\delta$ and $\gamma$ is $\pm1$,
so that in particular the homology class $[\delta]\in H_1(\Sigma,\bZ)$ is nonzero. Since $\gamma$ connected an arbitrary
pair of cusps, the first claim of the corollary follows.

If $\delta_1$ and $\delta_2$ encircle two distinct cusps of $\Sigma$, then since there are three distinct cusps in $\Sigma$,
one can choose two geodesics $\gamma_1$ and $\gamma_2$ as above so that
$\gamma_i$ pairs nontrivially only with $\delta_i$ under the Poincar\'e duality pairing, for $i\in\{1,2\}$. This shows that $\delta_1$ and
$\delta_2$ represent distinct integral homology classes, whence the second claim of the corollary follows.
\end{proof}

\subsection{Powers and discreteness}

We establish the following
relatively straightforward fact about discrete subgroups of $\pslr$ which will be useful at the end of the proof of our main result:

\begin{lem}\label{lem:square-discrete}
Let $H$ be a Zariski dense subgroup of $\pslr$ and suppose that there is a finitely generated discrete subgroup $\G<\pslr$ and an $N>0$
such that for all $h\in H$,
we have $h^N\in\G$. Then $H$ is discrete.
\end{lem}
\begin{proof}
If $H$ fails to be discrete then its topological
closure $\overline{H}$ must have positive dimension. Since $\pslr$ is simple and since $H$ is Zariski dense,
 we have that $\overline{H}$ is necessarily
equal to $\pslr$. It follows then that $H$ must be topologically dense in $\pslr$. Since the condition $|\tr A|<2$ is open for $A\in\pslr$, it
follows that $H$ must contain elements whose traces form a dense subset of $(-2,2)$. It follows that $H$ either contains an elliptic element
of infinite order or elliptic elements of arbitrarily high orders. In the first case, we obtain that $\G$ also contains an elliptic element of infinite
order, violating the discreteness of $\G$.

In the second case, choose a finite index subgroup $\G_0<\G$ which is torsion-free,
which exists by Selberg's Lemma~\cite{Raghunathan1972}. We therefore have a torsion element $A\in H$ and a power of $A$ that is
nontrivial and which lies in $\G_0$, a contradiction.
\end{proof}

\subsection{Commensurations of $\pslz$ and rational matrices}

In this section we record the following easy observation to which we have alluded already:

\begin{lem}\label{lem:rat-value}
Let \[g\in\pslq\sqQ=\comm(\pslz),\] and let $A\in\pslz$. Then $A^g\in\pslq$.
\end{lem}
What is technically meant by the statement of Lemma~\ref{lem:rat-value} is that for any representatives of $A$ and $g$ in
$\mathrm{SL}_2(\R)$, the
corresponding matrix $A^g$ will have rational entries.

\begin{proof}[Proof of Lemma~\ref{lem:rat-value}]
Choose a representative for $A$ in $\mathrm{SL}_2(\Z)$ and a representative $g=\sqrt{q}B$, where $q\in\Q$ and where
$B\in\mathrm{GL}_2(\Q)$. It is immediate that $B^{-1}AB$ has rational entries. Since the representative for $g$ differs from
$B$ by a scalar multiple of the identity, we see that $A^g=B^{-1}AB$ has rational entries.
\end{proof}

The fact that the commensurator of $\pslz$ requires the adjunction of square roots
and is not simply $\pslq$ is at times an annoying issue.
We note that the occurrence of square roots is fundamentally a vestige of the
fact that the center of
$\mathrm{SL}_2(\R)$ is nontrivial.

We remark that it is possible to avoid the appearance of square roots by working inside of other Lie groups.
For instance, one can consider
$\pslz$ as the group of
integer points of the special orthogonal group $\mathrm{SO}^+(f)$, where  \[f=xz-y^2.\] In this case,
$\mathrm{SO}(2,1)$ has no center,
and a general result of Borel~\cite{Borel66} implies that the commensurator of the integral points is simply the
group of rational points
$\mathrm{SO}^+(f,\Q)$. Thus, one avoids the complications resulting from square roots.

The cost of passing to the $3\times 3$ matrix
group is that for the upper half plane model or disk model of $\Hyp^2$, the group $\mathrm{SO}^+(f,\Q)$ is less intuitive and calculations are
more complicated. The relative simplicity of calculating inside of $\pslq$, together with the particular properties of $2\times 2$ matrices,
 will be an important feature of the arguments in the sequel,
especially in
the proof of Proposition~\ref{reid} and in
Section~\ref{sec:homology}. So, the potential gain from using a center--free Lie group is outweighed by computational and conceptual
complications later on. We therefore opt to retain the $2\times 2$ matrix group setting for the purposes of this article.

\section{Invariance under Commensuration}
Recall the notation that $\Phi(H)$ is a term of the lower central series or derived series of $H$. If $\Phi(H)\neq H$ then we sometimes
say that $\Phi(H)$ is a \emph{proper term}.
In this section, we prove that elements \[g\in\comm(\Phi(H))\] satisfying certain special properties
must lie in $\pslz$.

\subsection{Normalization of Zariski dense subgroups}\label{sec:invariant}

\begin{lem}\label{normalizer}
Let $\GG$ be a simple real Lie group and $\Ga \subset \GG$ be a Zariski dense discrete subgroup.
Then the normalizer $N_\GG(\Ga)$ is discrete.
\end{lem}

We restrict to real Lie groups, since there is a complication for complex Lie groups; namely, the unit complex numbers are Zariski dense
in $\C$, and are a positive dimensional closed subgroup of $\C$. In the case of complex Lie groups, one needs to consider unbounded Zariski
dense subgroups in order to avoid issues like this one. Since we are primarily interested in $\pslr$, we will simply rule out
the generality of arbitrary complex Lie groups.

\begin{proof}[Proof of Lemma~\ref{normalizer}]
Since $\Ga \subset \GG$ is Zariski dense, so is $N_\GG(\Ga)$. If $N_\GG(\Ga)$ is not discrete, then it must be all of $\GG$, since a positive dimensional Zariski dense Lie subgroup of a simple real Lie group $\GG$ is necessarily all of $\GG$. This then
 forces $\GG$ to admit a non-trivial Zariski dense normal subgroup $\Ga$, contradicting simplicity  of $\GG$. Hence $N_\GG(\Ga)$ is  discrete.
\end{proof}

As a consequence, we have the following:
\begin{lem}\label{lem:invariant}
Let $H\lhd\pslz$ be a finite index normal subgroup. Suppose that $g\in\pslr$ satisfies
$H^g=H$. Then $g\in\pslz$. More generally, if $\Phi(H)=\Phi(H)^g$ then $g\in\pslz$.
\end{lem}
\begin{proof} Let $\GG = \pslr$. Since $\pslr$ is simple and since $H$ is Zariski dense, it follows by Lemma \ref{normalizer} that
$N_\GG(H)$ is  discrete. We have that
	$H_1= \langle g,H\rangle$ satisfies $H_1 \subset N_\GG(H)$ and hence $H_1$ is discrete.

Since $H$ is normal in $\pslz$, we have \[\pslz<N_\GG(H).\]
It follows that $\langle g,\pslz\rangle$ is a discrete subgroup of
$\pslr$, by Lemma \ref{normalizer}. It follows that $\langle g,\pslz\rangle$ is a discrete group of orientation preserving isometries of
$\Hyp^2$.
A standard fact from hyperbolic geometry says that \[X=\Hyp^2/\pslz\] admits no further nontrivial orientation preserving isometries
and is therefore a minimal orbifold. Indeed, let $\iota$ be an orientation preserving involution of $X$, so that $X/\langle \iota\rangle$ is
the $(2,3,\infty)$ triangular orbifold $\TT$. The orbifold $\TT$ has no symmetries, and thus $X$ admits no orientation preserving symmetries,
as these would have descended to $\TT$. Consequently, we see that if \[\langle g,\pslz\rangle<\pslr\] is discrete then $g\in\pslz$.

The proof in the case where \[\Phi(H)=\Phi(H)^g\] is identical, since $\Phi(H)$ is Zariski dense in $\pslr$ and characteristic in $H$, and hence
normalized by $\pslz$.
\end{proof}

We note that it is at this point that we use the normality of $H$ in $\pslz$. If $H$ were not normal then we could pass to a finite-index subgroup
of $H$ which was normal, though then $g$ might no longer normalize $H$.


\subsection{An application of Chevalley-Weil theory}\label{sec:gaschutz}

In this section, we prove the following lemma:

\begin{lem}\label{lem:almost-inv}
Let $g\in\pslr$, let $H<\pslz$ have finite index, and suppose that there is a free group $F<\pslr$ such that $H,H^g<F$.
Suppose furthermore that $H$ is normal in $F$.
If $H^g\neq H$ then $g$ does not
commensurate $\Phi(H)$, provided that $\Phi(H)$ is proper.
\end{lem}

For Lemma~\ref{lem:almost-inv}, the ambient group $\pslr$ is irrelevant. We have the following general fact, from which
Lemma~\ref{lem:almost-inv} will follow with some more work.

\begin{lem}\label{lem:gaschutz}
Let $F$ be a finitely generated free group of rank at least two,
and let $K_1,K_2<F$ be distinct, isomorphic, finite index subgroups. Suppose furthermore that
$K_2$ is normal in $F$.
Then $K_1'$ and $K_2'$ are not
commensurable. That is to say, $K_1'\cap K_2'$ has infinite index in $K_1'$. If $K_1$ is also normal in $F$, then $K_1'\cap K_2'$
has infinite index in
$K_2'$ as well.
\end{lem}

The reader may be dissatisfied with the apparent asymmetry of Lemma \ref{lem:gaschutz}. In its application to commensurators of thin
groups, the asymmetry disappears, however. In deducing Lemma~\ref{lem:almost-inv} from Lemma~\ref{lem:gaschutz}, we set $H$
to be normal subgroup of finite index in a finite index free subgroup of $\pslz$,
which for the sake of explicitness we assume to be $\G(k)$ for some $k\geq 2$.
We assume furthermore that $H^g<\G(k)$. Of course, $H^g$ may fail 
to be normal in $\G(k)$. The conclusion of Lemma~\ref{lem:gaschutz} will imply (via Lemma~\ref{lem:filtration})
 that $\Phi(H)\cap \Phi(H^g)$ has infinite index in
$\Phi(H^g)$, but
says nothing about the index of $\Phi(H)\cap \Phi(H^g)$ in $\Phi(H)$.
However, if \[g\in\comm(\Phi(H))\] then we see that \[g^{-1}\in\comm(\Phi(H))\] as well, and so we obtain
that $\Phi(H)\cap \Phi(H^{g^{-1}})$ has infinite index in $\Phi(H)$, which symmetrizes the conclusion somewhat.
We note briefly that $\Phi(H^g)=\Phi(H)^g$.

Before proving Lemma~\ref{lem:gaschutz}, we recall the following classical fact about the homology of finite index subgroups of free groups.
This is also called Gasch\"utz's theorem  \cite{koberda-gd,gllw} and is a free-group version of a
well-known Theorem due to Chevalley and Weil \cite{cw}. We will not reproduce a proof of this result,
though we indicate that it can easily be deduced from the Lefschetz Fixed-Point Theorem:

\begin{theorem}\label{thm:gaschutz}
Let $F_k$ be a free group of rank $k$, and let $N\lhd F_k$ be a finite index normal subgroup with $Q=F_k/N$. Then as a $\Q[Q]$-module,
we have an isomorphism \[H_1(N,\Q)\cong \tau^k\oplus \rho^{k-1},\] where  $\tau$ denotes the trivial representation of $Q$, and 
$\rho=\rho_{reg}/\tau$ denotes the quotient of the regular representation by the trivial representation.
\end{theorem}

The trivial isotypic
component of $H_1(N,\Q)$ is canonically isomorphic to $H_1(F_k,\Q)$ via the transfer map. Note that if $z\in H_1(N,\Q)$ is not
in the image of the transfer map, then $z$ is not fixed by some element of $Q$. Moreover, if $1\neq q\in Q$, then there is an element
$z\in H_1(N,\Q)$ such that $q\cdot z\neq z$.

In this section and throughout the rest of the paper, when we refer to \emph{homology}, we always mean the first homology (with coefficients
that will be clear from context), unless otherwise noted.

\begin{proof}[Proof of Lemma~\ref{lem:gaschutz}]
Since $K_1$ and $K_2$ are distinct, isomorphic, and both of finite index in
$F$, there can be no inclusion relations between $K_1$ and $K_2$.
It follows that there exists an element $a\in K_1\setminus K_2$.
Let \[b\in K_1\cap K_2,\] and let $x_b=[a,b]$. Observe that \[x_b\in K_1'\cap K_2.\]
Indeed, since $a$ and $b$ are in $K_1$, we have that $x_b\in K_1'$.
From the fact that $x_b=b^{-1}b^a$ and the normality of $K_2$, we see that $x_b\in K_2$.

We may now compute the homology class $[x_b]$ of $x_b$, as an element of $H_1(K_2,\Z)$:
we have \[[x_b]=a\cdot [b]-[b],\] where $a\cdot [b]$ denotes the image of $[b]$
under the action of $a$, viewed as an element of $F/K_2$.

Since $a\notin K_2$, we have that $a$ represents a nontrivial element of $F/K_2$. It follows that there is a homology class
\[z\in H_1(K_2,\Z)\] such that $a\cdot z\neq z$, by Theorem~\ref{thm:gaschutz}. Note that replacing $z$ by a nonzero
integral multiple, we have \[a\cdot (nz)=n(a\cdot z)\neq nz.\]
We may therefore choose an $n\in\Z\setminus \{0\}$
and an element $b\in K_1\cap K_2$ such that $[b]=nz$, since $K_1$ and $K_2$ have finite index in $F$.

With such a choice of $b$, we have that \[[x_b]\in H_1(K_2,\Z)\] is a nontrivial homology class. Since $K_2$ is free, we see that
no power of $x_b$ represents a trivial homology class of $K_2$. Therefore, for all $N\neq 0$, we have $x_b^N\in K_1'$
but $x_b^N\notin K_2'$.
It follows that $K_1'\cap K_2'$ has infinite index in $K_1'$.

If $K_1$ is also normal in $F$, then one can switch the roles of $K_1$ and $K_2$ to conclude that $K_1'\cap K_2'$ has infinite index
in $K_2'$ as well.
\end{proof}

Note that Lemma~\ref{lem:gaschutz} establishes Lemma~\ref{lem:almost-inv} for $\Phi(H)=H'$. We prove the following
fact, which now immediately implies Lemma~\ref{lem:almost-inv} and which will be useful in the sequel:

\begin{lem}\label{lem:filtration}
Let $A$ and $B$ be commensurable, nonabelian,
free subgroups of finite rank
in an ambient group $G$. Write $\Phi(A)$ and $\Phi(B)$ for a proper term of the lower central series or derived series of $A$ and $B$,
so that if $\Phi(A)=\gamma_i(A)$ then $\Phi(B)=\gamma_i(B)$, or if $\Phi(A)=D_i(A)$ then $\Phi(B)=D_i(B)$. In either case, the index
$i\geq 2$
is the same for both $A$ and $B$.

Suppose there exists an element $g\in A\cap B$ such that
$g\in A'$ and such that $g\in B\setminus B'$. Then $\Phi(A)$ and $\Phi(B)$ are not commensurable.
\end{lem}

\begin{proof}
By the definition of $g$, we have that $g^N\in A'$ for all $N$ and $g^N\notin B'$ for all $N\neq 0$. We deal with the two series separately,
starting with the lower central series. The $n^{th}$ term of the lower central series of a group $H$  will be denoted by $\gamma_n(H)$,
 and the $n^{th}$ term of the derived series of a group $H$  will be denoted by $D_n(H)$.

Note that $g\in\gamma_2(A)$ and \[g\in\gamma_1(B)\setminus\gamma_2(B).\] Note also that there is an element $x\in A\cap B$ such that
$[g,x]\in\gamma_3(A)$ and such that \[[g,x]\in\gamma_2(B)\setminus\gamma_3(B).\] Indeed, it suffices to choose an element $x$ lying
in $A\cap B$ whose integral homology class in $B$ is nonzero, which exists since $A$ and $B$ are commensurable (cf.~\cite{MKS}).
Note that $[g,x]\in A\cap B$.

By an easy induction, we can find elements $\{y_n\}_{n\geq 1}\subset A\cap B$ such that
$y_n\in\gamma_{n+1}(A)$ and such that \[y_n\in\gamma_n(B)\setminus\gamma_{n+1}(B).\] Again, we can simply define
$y_{n+1}=[y_n,x]$, where $x\in A\cap B$ represents a nontrivial homology class of $B$.
For a finitely generated free group $F$, we
have that $\gamma_n(F)/\gamma_{n+1}(F)$ is a finitely generated torsion-free abelian group for all $n\geq 1$ (see again~\cite{MKS}).
It follows that for all $N\neq 0$, we have $y_n^N\in\gamma_{n+1}(A)$ and \[y_n^N\in\gamma_n(B)\setminus\gamma_{n+1}(B),\] whence
no proper terms of the lower central series of $A$ and $B$ are commensurable.

We now consider the derived series, and begin with the same element $g$ as above. We consider the groups $H_1(A',\Z)$ and $H_1(B',\Z)$,
both of which are infinitely generated free abelian groups. Since \[g\in A'=D_2(A),\]
we have that $g$ represents a (possibly trivial) element of
$H_1(A',\Z)$. Since $g\notin B'$, we have that $g$ represents a nontrivial element of $H_1(B,\Z)$. Thus, if $z\in H_1(B',\Z)$ is a nontrivial
homology class, then for all sufficiently large $N$ we have that $g^N\cdot z-z$ represents a nontrivial element of $H_1(B',\Z)$. This claim
is easily checked using covering space theory, by choosing a finite wedge of circles whose fundamental group is $B$ and taking the cover
of the wedge corresponding to $B'$. Then one simply uses the fact that $H_1(B,\Z)$ acts properly discontinuously on the corresponding cover.

Since $A$ and $B$ are commensurable, we can choose an element $x\in A'\cap B'$ such that $x$ represents a nontrivial element of
$H_1(B',\Z)$. Indeed, we can choose any two element $b_1,b_2$ in a free basis for $B$ and nonzero exponents $n_1$ and $n_2$
such that $b_i^{n_i}\in A$ for $i\in \{1,2\}$. Then, the element $x=[b_1^{n_1},b_2^{n_2}]$ will represent a nontrivial element of $H_1(B',\Z)$.

Note that for all $N$, we have \[y=[g^N,x]\in D_3(A).\] On the other hand, for $N$ sufficiently large we have that $y$ represents a nontrivial
element of $H_1(B',\Z)$, and therefore \[y\in D_2(B)\setminus D_3(B).\] Since $H_1(B',\Z)$ is torsion-free, it follows that $D_3(A)$ and
$D_3(B)$ are not commensurable.

An easy induction now shows that $D_i(A)$ and $D_i(B)$ are not commensurable for $i\geq 2$. Indeed, suppose we have produced
an element $y_i\in A\cap B$ such that for all nonzero $N$, we have $y_i^N\in D_i(A)$ and
\[y_i^N\in D_{i-1}(B)\setminus D_i(B).\] As in the case
$i=2$, it is straightforward to construct an element $x\in D_i(A)\cap D_i(B)$ such that $x$ represents a nontrivial element
of $H_1(D_i(B),\Z)$. Then for all $N$ we have \[y_{i+1}=[y_i^N,x]\in D_{i+1}(A).\] For all sufficiently large $N$, however,
we have that $y_{i+1}$
represents a nontrivial homology class in $H_1(D_i(B),\Z)$. Using again the fact that $H_1(D_i(B),\Z)$ is torsion-free, no power of
$y_{i+1}$ lies in $D_{i+1}(B)$, so that $D_{i+1}(A)$ and $D_{i+1}(B)$ are not commensurable.
\end{proof}



\section{Integral commensurators}\label{sec:intcomm}

In this section, we establish the following fact:

\begin{lem}\label{lem:integcomm} Let $H<\pslz$ be a non-elementary subgroup and
let $g\in\comm( H)$.  Then $g^2\in\pslq\sqQ$.
\end{lem}

Here, non-elementary simply means non-solvable. The proof of
Lemma \ref{lem:integcomm} is given in Section \ref{sec:ant} below and draws from the theory of invariant trace fields and quaternion division 
algebras. We are grateful to Alan Reid for teaching us about these ideas.

\subsection{Quaternion division algebras and commensurators}\label{sec:ant}
In this section, we give our first proof of Lemma~\ref{lem:integcomm}.

We refer the reader to \cite{mr} for the relevant basics on invariant trace fields and quaternion division algebras. The proof of \ref{reid} below is guided and informed  by the argument on page 118 of \cite{mr} proving Theorem 3.3.4 there (see especially Equations 3.8 and 3.9).

\begin{prop}\label{reid} {\bf (Reid)}
	Let  $H$ be a non-elementary (not necessarily discrete) subgroup of $\pslc$ such that \[K:=\Q(\tr H) = \Q(\tr H')\] for any 
	subgroup $H'$ of finite index in $H$.
Let $B=A_0H$ denote the quaternion algebra generated over $K$. That is, $B$ is obtained by taking finite linear combinations of elements of $H$ over $K$~\cite[Sec 3.2]{mr}. Let $ B^*$ denote the set of invertible elements of $B$.
Suppose that $H$ and $xHx^{-1}$ are 
commensurable, i.e. $x\in {\rm Comm}_\pslc (H)$. Then the following conclusions hold:
\begin{enumerate}
\item $ x= ta$ where $a\in B^*$ and $t$ is a non-zero complex number.
\item $t^2\in K$ and so $x^2\in B^*$. 
\end{enumerate}
\end{prop}

\begin{proof}[Proof of Lemma \ref{lem:integcomm} assuming Proposition \ref{reid}]
Note that the hypothesis \[\Q=\Q(\tr H) = \Q(\tr H')\] in Proposition \ref{reid} is satisfied for arbitrary subgroups of $\pslz$.
Moreover, we have that \[B^*=\pslq\sqQ\] in our notation, whence the desired conclusion follows.
\end{proof}

We now turn to the proof of  Proposition \ref{reid}:

\begin{proof}[Proof of Proposition~\ref{reid}]
Let \[H_1=H\cap xHx^{-1},\] so that $H_1$ has finite index in $H$ and $xHx^{-1}$. 
Hence, \[\Q(\tr H_1)=\Q(\tr H)=\Q(\tr(xHx^{-1}))\] by hypothesis.
Hence the quaternion algebras $B=A_0H$, \[A_0(xHx^{-1})=xA_0Hx^{-1},\] and 
$A_0H_1$ are all defined over $K$ and hence
 are all equal.

It follows from the Skolem--Noether Theorem \cite[2.9.8]{mr} there exists 
$a\in B^*$ such that 
\begin{equation}
aga^{-1}=xgx^{-1}
\end{equation}
 for all $g\in B$. Thus 
 \begin{equation}\label{eq:conj}
 a^{-1}x g 
=ga^{-1}x
\end{equation} for all $g \in B$.

We would like to conclude that \[a^{-1}x\in Z(B)=K,\] but we cannot immediately do this
because $a^{-1}x$ need not be in $B$.
However, after  tensoring  with  $\C$ over $K$, the Equation \ref{eq:conj} continues to hold: $$a^{-1}x 
g =ga^{-1}x$$ for all $g\in M_2(\C)$.
Hence, \[a^{-1}x = t \in\C\] is a non-zero element and $x=ta$ as required, which proves the first conclusion.

Finally, the equation \[\det(x)=\det(ta)\] gives us \[1=t^2\det(a),\] and we have $\det(a) \in K$ by assumption. 
It follows that $t^2\in K$, and so \[x^2=t^2a^2\in B^*,\]
which establishes the second part of the conclusion.
\end{proof}

\section{Homological pseudo-actions and completing the proof} \label{sec:homology}

In this section, we complete the proof of Theorem~\ref{thm:main}. Let $H\lhd\pslz$ be a finite index normal subgroup which is contained in
the principal congruence subgroup $\G(k)$. We have shown that if \[g\in\comm(\Phi(H))\] then $g^2\in\pslq\sqQ$
and hence \[g^2\in\comm(\pslz).\]
Therefore, we have that $H\cap H^{g^2}$ is a finite index subgroup of both $H$ and $H^{g^2}$.
We wish to argue that $H^{g^2}<\G(k)$ as well,
so that we can apply Lemma~\ref{lem:invariant} or Lemma~\ref{lem:almost-inv}, depending on whether $H=H^{g^2}$ or not.
Once we achieve this goal, we can prove the main technical result of this section:

\begin{lem}\label{lem:pseudo-para}
Let $H<\pslz$ be a finite index normal subgroup
which is contained in $\G(k)$ for some $k\geq 2$, and let $g\in\pslq\sqQ$ commensurate $\Phi(H)$ for some proper term.
Then $g\in\pslz$.
\end{lem}

The proof of Theorem~\ref{thm:main} follows quickly from Lemma~\ref{lem:pseudo-para}.

\begin{proof}[Proof of Theorem~\ref{thm:main}, assuming Lemma~\ref{lem:pseudo-para}]
Let $H$ be as in the statement of the theorem and let $g\in\comm(\Phi(H))$.
By Lemma \ref{lem:integcomm}, we see that \[g^2\in\pslq\sqQ.\]
By Lemma~\ref{lem:pseudo-para}, we have that $g^2\in\pslz$. It follows
that the square of every element in \[\comm(\Phi(H))\] lies in $\pslz$, so that the group \[\comm(\Phi(H))\]
is discrete by Lemma~\ref{lem:square-discrete}.
\end{proof}

Thus, it remains to establish Lemma~\ref{lem:pseudo-para}, which will occupy the remainder of this section.

\subsection{Building homological pseudo-actions}\label{sec:pseudo}
Note that $H$ and $H^{g^2}$ both lie in $\pslq$, as is checked by an easy computation. Indeed, $H$
is a subgroup of $\pslz$ and is therefore a subgroup of $\pslq$. The group
$H^{g^2}$ lies in $\pslq$ by Lemma~\ref{lem:rat-value}, since $g^2\in
\pslq\sqQ$.
Let $H=\langle x_1,\ldots,x_m\rangle$, where $\{x_1,\ldots,x_n\}$ is a
 free basis for $H$. We write $[x_i]$ for the homology class of $x_i$. We have that \[\{[x_1],\ldots,[x_n]\}\] generate the integral homology
of $H$, and for each $N\geq 1$, we have that \[\{[x_1^N],\ldots,[x_n^N]\}\] generate a finite index subgroup of the integral homology of $H$.
Suppressing $N$ from the notation, we sometimes write $z_i=[x_i^N]$. For $y\in H^{g^2}$ arbitrary,
we consider the homology class of the commutator $[y,x_i^N]$, when it makes sense.
Since $y\in H^{g^2}$ and since $H$ and $H^{g^2}$ are commensurable,
there exist arbitrarily large values of $N$ for which \[[y,x_i^N]\in (H^{g^2})',\] since
we merely choose values of $N$ such that $x_i^N\in H\cap H^{g^2}$.

On the other hand, we have that $x_i^N\in H$ for all $i$ and $N$. Similarly, there exist arbitrarily large values of $N$ such that
$(x_i^N)^y\in H$ as well, independently of $i$, since $y\in\pslq$ and hence $H$ and $H^y$ are commensurable.
For these values of $N$,
we may make sense of the homology class $y\cdot z_i-z_i$ for all $i$, which is the homology class of $[y,x_i^N]$
in $H$. This defines a \emph{pseudo-action} of $H^{g^2}$
on the integral homology of $H$. We call it a  pseudo-action since it is not defined for all $N$.

Suppose that $y\cdot z_i-z_i$ is nonzero for some $i$. Then we obtain \[[y,x_i^N]\in H\setminus H'.\] Then by
Lemma~\ref{lem:filtration}, we see that the group $\Phi(H)$ and the group \[\Phi(H^{g^2})=\Phi(H)^{g^2}\] are not commensurable.
Note that the previous discussion was symmetric in $H$ and $H^{g^2}$, so that we obtain elements of $H^{g^2}\cap H'$, no powers of
which lie in $(H^{g^2})'$ unless the homological pseudo-action of $H$ on the integral homology of $H^{g^2}$ is trivial.

Therefore, if $g$ commensurates $\Phi(H)$ for some proper term of the lower central series or derived series of $H$,
we must have that $y\cdot z_i-z_i$ is the trivial integral homology class of $H$ for all $i$, whenever this homology class is defined.
In particular, the pseudo-action of $y$ on the integral homology of $H$ is trivial.

\subsection{Trivial homology pseudo-actions and parabolics}\label{sec:parabolics}

Let $\gamma\in H$ be a parabolic element fixing infinity, which exists because $H$ has finite index in $\pslz$. 
Let $y\in H^{g^2}$ be arbitrary, and suppose that the $y$-pseudo-action
on the integral homology of $H$ is trivial. 

We see that there is a positive integer $N$ such that the homology class of $\gamma^N$ is
invariant under $y$. That is, \[[(\gamma^N)^y]=[\gamma^N]\] as homology classes of $\Hyp^2/H$.
Since $H<\G(k)$ for some $k\geq 2$, we see that each cusp of
$\Hyp^2/H$ is homologically nontrivial, and no two distinct cusps represent the same homology class, as follows from 
Corollary \ref{cor:homology}.

It follows that the element $(\gamma^N)^y$ represents a power of a free homotopy class of a cusp of $\Hyp^2/H$,
which is equal to the free homotopy class represented by $\gamma^N$.
In particular, we have that $(\gamma^N)^y$ is a parabolic element of $\pslq$, and the
fixed point of $(\gamma^N)^y$ is in the $H$-orbit of infinity.

It follows that there is an element $h\in H$ such that \[((\gamma^N)^y)^h=(\gamma^N)^{yh}\] stabilizes infinity.
Since both $H$ and $H^{g^2}$ are subgroups of $\pslq$, it follows that
$yh$ lies in the stabilizer of infinity in $\pslq$, so that we
have \[yh=\begin{pmatrix}r&t\\0&r^{-1}\end{pmatrix}\]
for some suitable $r,t\in\Q$.

Writing
\[\gamma^N=\begin{pmatrix}1&M\\0&1\end{pmatrix}\] for some suitable $M\in\Z$, then we see that
\[(\gamma^N)^{yh}=\begin{pmatrix}1&r^{-2}M\\0&1\end{pmatrix}.\] Since \[[(\gamma^N)^{yh}]=[(\gamma^N)^y]=[\gamma^N],\]
it follows that we must have $r=1$,
so that \[yh=\begin{pmatrix}1&t\\0&1\end{pmatrix}.\]

Repeating the same argument for a parabolic element
of $H$ stabilizing $0$, we find an element $q\in H$ such that \[yq=\begin{pmatrix}1&0\\s&1\end{pmatrix}\] for some suitable
$s\in\Q$. We now multiply $(yh)^{-1}$ and $(yq)$ to get \[\begin{pmatrix}1-ts&-t\\s&1\end{pmatrix}\in H.\] Since $H<\G(k)$,  we
see that $t,s\in k\Z$. It follows that $y\in\G(k)$. Since $y\in H^{g^2}$ was chosen arbitrarily, we see that $H^{g^2}<\G(k)$.

To conclude the discussion of the last two subsections, we can finally prove Lemma~\ref{lem:pseudo-para}, which
completes the proof of Theorem~\ref{thm:main}.

\begin{proof}[Proof of Lemma~\ref{lem:pseudo-para}]
The discussion in Sections~\ref{sec:pseudo} and~\ref{sec:parabolics} implies that $H^{g}<\G(k)$. Note here that we conjugate by $g$ and
not by $g^2$, since we already assume that $g\in\pslq\sqQ$.

Note that $H$ is normal in $\pslz$ and therefore is normal in $\G(k)$ as well.
Since $\G(k)$ is a free group
and since $H$ and $H^{g}$ are isomorphic, if $H\neq H^{g}$ then we can apply Lemma~\ref{lem:almost-inv} to conclude that
$\Phi(H)$ and $\Phi(H)^{g}$ are not commensurable. If $H=H^{g}$ then Lemma~\ref{lem:invariant} implies that $g\in\pslz$.
\end{proof}


\begin{thebibliography}{10}
\bibliographystyle{amsplain}

\bibitem{benoistgafa}
Y.~Benoist.
\newblock {Propri{\'e}t{\'e}s asymptotiques des groupes lin{\'e}aires.}
\newblock {\em Geom. Funct. Anal., 7(1)}, pages 1--47, 1997.

\bibitem{biswasmj}
K.~Biswas and M.~Mj,
\newblock Pattern rigidity in hyperbolic spaces: duality and {PD} subgroups.
\newblock {\em Groups Geom. Dyn.}, 6(1):97--123, 2012.

\bibitem{Borel66}
A.~Borel.
\newblock Density and maximality of arithmetic subgroups
\newblock {\em J. Reine Angew. Math.} 224 (1966), 78--89.

\bibitem{cw}
C.~Chevalley, A.~Weil, and E.~Hecke.
\newblock \"{U}ber das verhalten der integrale 1. gattung bei automorphismen
  des funktionenk\"{o}rpers.
\newblock {\em Abh. Math. Sem. Univ. Hamburg}, 10(1):358--361, 1934.

\bibitem{gllw}
F.~, M.~Larsen, A.~Lubotzky, and J.~Malestein.
\newblock Arithmetic quotients of the mapping class group.
\newblock {\em Geom. Funct. Anal.}, 25(5):1493--1542, 2015.

\bibitem{greenberg}
L.~Greenberg.
\newblock Commensurable groups of {M}oebius transformations.
\newblock pages 227--237. Ann. of Math. Studies, No. 79, 1974.

\bibitem{koberda-gd}
T.~Koberda,
\newblock Asymptotic linearity of the mapping class group and a homological
  version of the {N}ielsen-{T}hurston classification.
\newblock {\em Geom. Dedicata}, 156:13--30, 2012.

\bibitem{llr}
C.~J. Leininger, D.~D. Long, and A.~W. Reid.
\newblock {Commensurators of non-free finitely generated Kleinian groups}.
\newblock {\em Algebraic and Geometric Topology 11}, pages 605--624, 2011.

\bibitem{loeffler}
D. Loeffler,
\newblock {Modular forms course},

\newblock \url{https://homepages.warwick.ac.uk/~masiao/modforms/solutions2.pdf},
2008


\bibitem{MKS}
W.~Magnus, A.~Karrass, and D.~Solitar,
\newblock Combinatorial group theory: Presentations of groups in terms of generators and relations.
\newblock {\em Interscience Publishers}, 1966.

\bibitem{mr}
C.~Maclachlan and A.~W.~Reid, 
\newblock \emph{The arithmetic of hyperbolic 3-manifolds}.
\newblock Springer-Verlag, New York, 2003, Graduate Texts in Mathematics,
219.

\bibitem{margulis}
G.~A. Margulis,
\newblock Discrete {S}ubgroups of {S}emisimple {L}ie {G}roups.
\newblock {\em Springer}, 1990.

\bibitem{mahan-commens}
M.~Mj,
\newblock {On Discreteness of Commensurators}.
\newblock {\em Geom. Topol. 15}, pages 331--350, 2011.

\bibitem{mahan-pattern}
M.~Mj,
\newblock Pattern rigidity and the {H}ilbert-{S}mith conjecture.
\newblock {\em Geom. Topol.}, 16(2):1205--1246, 2012.

\bibitem{morris-book}
D.~Witte Morris,
\newblock {\em Introduction to arithmetic groups}.
\newblock Deductive Press, 2015.

\bibitem{Raghunathan1972}
M.~S. Raghunathan,
\newblock \emph{Discrete subgroups of {L}ie groups}.
\newblock Springer-Verlag, New York-Heidelberg, 1972, Ergebnisse der Mathematik und ihrer Grenzgebiete,
  Band 68.

\bibitem{sarnak-thin}
P.~Sarnak,
\newblock Notes on thin matrix groups.
\newblock In {\em Thin groups and superstrong approximation}, volume~61 of {\em
  Math. Sci. Res. Inst. Publ.}, pages 343--362. Cambridge Univ. Press,
  Cambridge, 2014.
  
  \bibitem{schwartz-pihes}
  R.~E. Schwartz,
  \newblock The quasi-isometry classification of rank one lattices.
  \newblock {\em Inst. Hautes \'{E}tudes Sci. Publ. Math.}, (82):133--168 (1996),
  1995.
  
  \bibitem{schwartz-inv}
  R.~E. Schwartz,
  \newblock Symmetric patterns of geodesics and automorphisms of surface groups.
  \newblock {\em Invent. math.} 128, (1997) No. 1, 177--199.


\end{thebibliography}
\end{document}